\documentclass{amsart}
\usepackage{amssymb,latexsym}
\usepackage{amsmath}
\usepackage[dvipdfm]{graphicx}
\usepackage{color}
\usepackage{enumerate}
\newenvironment{enumeratei}{\begin{enumerate}[\upshape (i)]} {\end{enumerate}}
\numberwithin{equation}{section}
\newcommand \gyemant {$\pmb\lozenge$}
\theoremstyle{plain}
 \newtheorem{theorem}{Theorem}[section]

 \newtheorem{lemma}[theorem]{Lemma}

\theoremstyle{definition}

 \newtheorem{remark}[theorem]{Remark}
\theoremstyle{remark}

 \newtheorem*{ackno}{Acknowledgment}
\theoremstyle{plain} 
\theoremstyle{definition} 
\theoremstyle{remark} 

\newcommand\semmi[1] {}
\newcommand\piros[1]{{\textcolor{red}{#1}}}

\newcommand\tudnivalo [1] {}  

\newcommand \temphint [1]{}

\newcommand \skinfo [1] {}

\newcommand \ssty [1] {\scriptstyle{#1}}
\newcommand \sssty [1] {\scriptscriptstyle{#1}}
\newcommand \tbf[1] {\textbf{#1}}

\renewcommand\phi{\varphi}
\renewcommand\rho{\varrho}   
\renewcommand\epsilon{\varepsilon} 
\renewcommand \theta {\vartheta}   

\newcommand \Jir [1] {\textup{Ji}\,#1} 

\newcommand \length [1] {\textup{length}\,#1}
\newcommand\ideal[1]{\mathord\downarrow #1}

\newcommand \leftb [1]  {\textup{BC}_{{\ell}}(#1)} 

\newcommand \cornl [1] { w_{#1}^\ell } 
\newcommand \cornr [1] { w_{#1}^r }    

\newcommand \ilupphold [2] {\textup{sp}_{\sssty{\textup{left}}}^{\kern-2pt +\kern 2 pt #2}(#1)} 
\newcommand \irupphold [2] {\textup{sp}_{\sssty{\textup{right}}}^{\kern-2pt +\kern 2pt #2}(#1)} 
\newcommand \uucover [1] {#1^{\mathord{+}}} 
\newcommand \ulcover [1] {#1^{\mathord{-}}} 

\newcommand \smLatsign {\textup{SSL}}
\newcommand \smDiagsign {\textup{SSD}}
\newcommand \multi [1] {{\ddot{#1}}}
\newcommand \multismLatsign {\textup{SS}\multi{\textup{L}}}

\newcommand \dsmLat[1] {\smLatsign{}^{\kern-1pt\delta\kern-1pt}(#1)}
\newcommand \dpsmLat[1] {\kern-1pt{\multismLatsign}{}^{\kern-1pt\delta\kern-1pt}(#1)}

\newcommand \dismLat[1] {\smLatsign{}^{\kern-1pt\delta\kern-1pt}(#1)\fentjelkong1}
\newcommand \dipsmLat[1] {{\multismLatsign}{}^{\kern-1pt\delta\kern-1pt}(#1)\fentjelkong1}
\newcommand \fentjelkong[1]{{}^{\kern- #1pt \mathord{\cong}}}
\newcommand \fentjelsim[1]{{}^{\kern- #1pt \mathord{\sim}}}

\newcommand \sampleclssign {\mathcal K}

\newcommand \dsampleclasps[1]  {\sampleclssign{}^{\kern0pt\delta}(#1)}
\newcommand \disampleclasps[1] {\sampleclssign{}^{\kern0pt\delta}(#1)\fentjelkong0}

\newcommand \LatCompSerSign {\textup{CSL}}
\newcommand \multiLatCompSerSign {\textup{CS}\multi{\textup{L}}}

\newcommand \dLcs  [1] {\LatCompSerSign{}^{\kern0pt\delta}(#1)}
\newcommand \dpLcs [1] {{}^{\kern0pt\delta}\kern-1pt\multi{\LatCompSerSign}(#1)}

\newcommand \diLcs [1] {\LatCompSerSign{}^{\kern0pt\delta}(#1)\fentjelkong1}
\newcommand \dipLcs[1] {{\multiLatCompSerSign}{}^{\kern0pt\delta}(#1)\fentjelkong1}
\newcommand \pblock [1] {[#1]^{\kern-1pt\mathord{\sim}}}
\newcommand \simfactor [1] {#1/{\kern-1pt\mathord{\sim}}}

\newcommand \Bpernik [2] {P^{\kern-1pt\mathord{\sim}}(#1,#2)} 
\newcommand \bperniksign {p^{\kern-1pt\mathord{\sim}}} 
\newcommand \Bhpernik [4] {P^{\kern-1pt\mathord{\sim}}_{#1,#2}(#3,#4)} 
\newcommand \bhpernik [4] {p^{\kern-1pt\mathord{\sim}}_{#1,#2}(#3,#4)} 
\newcommand \Ipernik [2] {\widehat{\kern 1pt P}(#1,#2)} 
\newcommand \iperniksign{\widehat {\kern 1pt p} } 
\newcommand \Itpernik [2] {\widehat{\kern 1pt I}(#1,#2)} 
\newcommand \itperniksign {\widehat {\kern 1pt i}}

\newcommand \restrict [2] {{#1}\kern-1pt \rceil_{\kern-1pt #2}}
\newcommand\set [1]{\{#1\}}

\newcommand \celldn [1] {\textup{cell}_{\kern-0.3pt\sssty{\mathord{\lozenge}}}(#1)}
\newcommand \cellup [1] {\textup{cell}^{\kern-0.7pt\sssty{\mathord{\lozenge}}}(#1)}
\newcommand \cgjcon [1] {\textup{con}_{\vee}\kern-1pt(#1)} 

\newcommand \boper [1] {#1^{\kern-1pt\ssty{{\mathord=\kern-5pt\mathord{\parallel}}}}}

\newcommand \kernit {{\kern -6pt}}
\newcommand \tbu[1] {{ \phantom{\Big|} \kern -2pt \boxed{\tbf{#1}}  }}  
\newcommand \tbw[1] {{ \phantom{\Big|} \kern -2pt \boxed{\tbf{#1}}\kernit}}  

\newcommand \secie {\boldsymbol{\rho}_{\kern -1pt e}^i}
 
\newcommand \numssl [1] {N_{\textup{ssl}}(#1)}
\newcommand \numssd [1] {N_{\textup{ssd}}(#1)}

\newcommand \bpbeta[1] {\boldsymbol{\beta}_{\kern-1pt #1}}

\newcommand \joing {\mathrel{\mathord\vee_{\kern -2pt G}}}
\newcommand \joinl {\mathrel{\mathord\vee_{\kern -2pt L}}}
\newcommand \joinvg {\mathrel{\mathord\vee_{\kern -2pt G'}}}
\newcommand \diai [1] {#1^{{\kern-0.7pt\natural}}}
\newcommand \diah [1] {#1{}^{{\kern-0.7pt  \sssty{{\triangledown}} }}}

\newcommand \diag [1] {#1{}^{{\kern-0.7pt\ast}}}
\newcommand \diac [1] {(#1)^{{\kern-0.7pt\diamond}}}
\newcommand \diavarc [1] {#1^{{\kern-0.7pt\diamond}}}

\newcommand \semipatch  {{ \,\pmb{\mathcal H}\kern 0.5pt}} 
\newcommand \patch [1] { \pmb{\mathcal P}_{\kern-2pt\textup{max}}(#1)} 
\newcommand \aslim [1] {#1^{\kern-1pt \bullet}} 

\newcommand \cproj {\mathrel{ \mathord{\Rightarrow} \kern-7.5pt \mathord{\Rightarrow} }} 
\newcommand \cpreq {\mathrel{ \mathord{\Leftarrow}  \kern-7.5pt \mathord{\Leftrightarrow} \kern-7.5pt \mathord{\Rightarrow} }} 
\newcommand \uppers {\,{\buildrel{\sssty{\textup{up}}}\over \rightarrow}\kern-8pt\mathord{\rightarrow}\; } 
\newcommand \upperpa [1] {\,{\buildrel{\sssty{\textup{up}}}\over \rightarrow}\kern-8pt\mathord{\rightarrow}_{#1}\; } 
\newcommand \dnpers {\,{\buildrel{\sssty{\textup{dn}}}\over \rightarrow}\kern-8pt\mathord{\rightarrow}\; } 
\newcommand \dnperpa [1] {\,{\buildrel{\sssty{\textup{dn}}}\over \rightarrow}\kern-8pt\mathord{\rightarrow}_{#1}\; } 
\newcommand \drestrict [2] {{#1}\rceil_{\kern-1pt #2}} 

\newcommand \url [1] {\texttt{#1}}
\newcommand \konst{C}
\newcommand \brang[1]{\textup{rank}_\ell(#1)}
\newcommand \jrang[1]{\textup{rank}_r(#1)}
\newcommand \ssdset[1] {\smDiagsign(#1)}
\newcommand \ssdnulset[1] {\smDiagsign_{00}(#1)}
\newcommand \ssdpozset[1] {\smDiagsign_{\textup{++}}(#1)}
\newcommand \ssdeeset[1] {\smDiagsign_{11}(#1)}
\newcommand \csonkad {D^\ast}
\newcommand \upintp [1] {\lceil #1\rceil}
\newcommand \upintgy [1] {\upintp{\sqrt{#1}\,}}
\newcommand \cissor [1]  {#1^{\blacktriangleleft}}

\newcommand \rspace {\kern 0.8pt}

\begin{document}
\title[Asymptotic number of slim, semimodular  diagrams]
{The asymptotic number of planar, slim, semimodular lattice diagrams}
\author[G.\ Cz\'edli]{G\'abor Cz\'edli}
\email{czedli@math.u-szeged.hu}
\urladdr{http://www.math.u-szeged.hu/$\sim$czedli/}
\address{University of Szeged\\Bolyai Institute\\
Szeged, Aradi v\'ertan\'uk tere 1\\HUNGARY 6720}


\thanks{This research was supported by the NFSR of Hungary (OTKA), grant numbers   K77432 and K83219, and by T\'AMOP-4.2.1/B-09/1/KONV-2010-0005}



\date{June 16, 2012}

\begin{abstract} A lattice $L$ is \emph{slim} if it is finite and the set of its join-irreducible elements contains no three-element antichain. 
We prove that there exists a positive constant $\konst$ such that, up to
similarity,  the number of planar diagrams of these lattices of size $n$ is asymptotically 
$\konst\cdot 2^n$.

\end{abstract}

\maketitle

\section{Introduction and the result}\label{section:intRo}
A finite lattice $L$ is \emph{slim} if $\Jir L$, the set of 
join-irreducible elements of $L$, contains no three-element antichain.
Equivalently, $L$ is slim if $\Jir L$ is the union of two chains.  Slim, semimodular lattices were heavily used while proving a recent generalization of the classical Jordan-H\"older theorem for groups in \cite{r:czg-sch-JH}.
These lattices are \emph{planar}, that is, they have planar diagrams, see \cite{r:czg-sch-JH}. Hence it is reasonable to study their planar diagrams, which are called \emph{slim, semimodular $($lattice$)$ diagrams} for short. The \emph{size} of a diagram is the number of elements of the lattice it represents. Let $D_1$ and $D_2$ be two planar lattice diagrams. A bijection $\phi\colon D_1\to D_2$ is a \emph{similarity map} if it is a lattice isomorphism preserving the left-right order of (upper) covers and that of lower covers of each element of $D_1$. If there is a similarity map $D_1\to D_2$, then these two diagrams are  \emph{similar}, and we will treat them as equal ones. 
Let $\numssd n$ denote the number of slim, semimodular diagrams of size $n$, counting them up to similarity. 
Our target is to prove the following result. 

\begin{theorem}\label{thmmain}
There exists a constant  $\konst$ such that $0<\konst<1$ and $\numssd n$ is asymptotically $\konst\cdot 2^n$, that is, $\lim_{n\to\infty}\bigl(\numssd n/ 2^n\bigr)=\konst$.
\end{theorem}

Note that there are two different methods to deal with $\numssd n$. The present one yields the asymptotic statement above, while the method of \cite{r:czgandfour} gives the exact values of $\numssd n$ up to $n=50$ (with the help of a usual personal computer). Also,  \cite{r:czgandfour}  determines the number $\numssl n$ of  slim, semimodular \emph{lattices} of size $n$ up to $n=50$ while we do not even know $\lim_{n\to\infty}\bigl(\numssl n/\numssl{n-1} \bigr)$, and it is only a conjecture that this limit exists.

Note also that, besides \cite{r:czgandfour} and \cite{r:czgolub}, there are several papers on counting  lattices; see, for example, M.~Ern\'e, J.~Heitzig, and J.~Reinhold~\cite{r:erneheitzigreinhold}, M.\,M.~Pawar and B.\,N.~Waphare~\cite{r:PawarWaphare}, and J.~Heitzig and J.~Reinhold~\cite{r:heitzigreinhold}.

\section{Lattice theoretic lemmas}

A minimal non-chain region of a planar lattice  diagram $D$ is called a \emph{cell}, a four-element cell is a \emph{$4$-cell}; it is also a \emph{covering square}, that is, cover-preserving four-element Boolean
sublattice. We say that  $D$ is  a \emph{$4$-cell diagram} if all of its cells are 4-cells. 
We shall heavily rely on the following result of G.~Gr\"atzer and E.~Knapp~\cite[Lemmas 4 and 5]{GK01}.

\begin{lemma}\label{lemmagrkn}
Let $D$ be a finite, planar lattice diagram.
\begin{enumeratei}
\item If $D$  is semimodular, then it is a $4$-cell diagram. If  $A$ and $B$
are $4$-cells of $D$ with the same bottom, then these $4$-cells have the same
top.
\item If $D$ is a $4$-cell diagram in which no two $4$-cells with the same
bottom have distinct tops, then $D$ is semimodular.
\end{enumeratei}
\end{lemma}

In what follows, we always assume that $4\leq n\in\mathbb N=\set{1,2,\ldots}$, and that $D$ is a slim, semimodular diagram of size $n$. Let $\cornl  D$ be the smallest doubly irreducible element of the left boundary chain 
$\leftb D$ of $D$, and let $\brang D$ be the height of $\cornl  D$.The left-right duals of these concepts are denoted by $\cornr  D$ and $\jrang D$. See Figure~\ref{fig:egy} for an illustration, where $\cornl D$ and $\cornr D$ are the black-filled elements. 
By D.~Kelly and I.~Rival \cite[Proposition 2.2]{kellyrival}, each planar lattice diagram with at least three elements contains a doubly irreducible element $\neq 0,1$ on its left boundary. This implies the following statement, on which we will rely implicitly.

\begin{figure}
\centerline
{\includegraphics[scale=0.9]{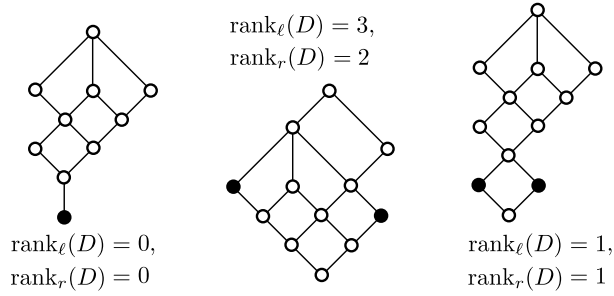}}
\caption{Left and right ranks  \label{fig:egy}}
\end{figure}

\begin{lemma} Either $\brang D=\jrang D=0$ and $\cornl  D=\cornr  D=0$, or 
$\brang D>0$ and $\jrang D>0$.
\end{lemma}

For $a\in D$, the ideal $\set{x\in D:x\leq a}$ is denoted by $\ideal a$.

\begin{lemma}\label{wlalattnoredlMa} 
$\leftb D\cap\ideal{\cornl D}\subseteq \Jir D$. 
\end{lemma}

\begin{proof}
Suppose, for a contradiction, that the lemma fails, and let $u$ be the smallest join-reducible element belonging to  $\leftb D\cap\ideal{\cornl D}$. 
 By D.~Kelly and I.~Rival \cite[Proposition 2.2]{kellyrival}, there is a doubly irreducible element $v$ of the ideal $\ideal u=\set{x\in D: x\leq u}$ such that 
$v\in \leftb{\ideal u}$; notice that  $v$ also belongs to  $\leftb D$. 
Clearly, $v<u$ and $v$ is join-irreducible in $D$. Therefore, since $v<u\leq \cornl D$ and $\cornl D$ is the least doubly irreducible element  of $\leftb D$, $v$ is meet-reducible in $D$. 
Hence there exist a $p\in D$ such that $v\prec p$ and $p\notin \ideal u$. Denote by $u_0$ the unique lower cover of $u$ in $\leftb D$. Since $v<u$, we have that $v\leq u_0$. By semimodularity and $p\not\leq u_0$, we obtain that $u_0=u_0\vee v\prec u_0\vee p\neq u$. Hence $u_0$ has two covers, $u$ and  $u_0\vee p$. Thus $u_0,u\in\leftb D$, $u_0\prec u$, $u$ is join-reducible, and $u_0$ is meet-reducible. This contradicts \cite[Lemma 4]{r:czg-sch-visual}.
\end{proof}

Next, we prove the following lemma.

\begin{lemma}\label{lMasieaeh}
 For $4\leq n\in\mathbb N$, we have that
\begin{align}
\numssd{n-1}+\numssd{n-3}&\leq \numssd n\label{kregy},\\
\numssd n&\leq 2\cdot\numssd{n-1}\text.\label{krket}
\end{align}
\end{lemma}

\begin{proof}
The set of slim, semimodular diagrams of size $n$ is denoted by $\ssdset n$. Let 
\begin{align*}\ssdnulset n&=\set{D\in \ssdset n: \brang D=\jrang D=0},\cr
\ssdeeset n &= \set{D\in \ssdset n: \brang D=\jrang D=1},\text{ and}\cr
\ssdpozset n &= \ssdset n - \ssdnulset n
\text.
\end{align*}
Since we can omit the least element and the least three elements, respectively, and the remaining diagram is still slim and semimodular by Lemma~\ref{lemmagrkn}, we conclude that $|\ssdnulset n|=\numssd{n-1}$ and 
$|\ssdeeset n|=\numssd{n-3}$. This implies \eqref{kregy}.
For $D\in \ssdpozset n$, we define 
\[\csonkad=D-\set{\cornl D}\text.
\]
We know from By D.~Kelly and I.~Rival \cite[Proposition 2.2]{kellyrival}, mentioned earlier, that 
\begin{equation}\label{sarokhely}
\cornl D\notin\set{0,1}\text{, \ provided }D\in\ssdpozset n\text.
\end{equation}
This together with the fact that $D\in \ssdpozset n$ is not a chain yields that 
\begin{equation}\label{hosszmarad}
\length \csonkad =\length D\text.
\end{equation}
Let $\ulcover{{\cornl D}}$ denote the unique lower cover of $\cornl D$ in $D$.
Since each meet-reducible element has exactly two covers by \cite[Lemma 2]{r:czg-sch-visual}, we conclude from Lemma~\ref{wlalattnoredlMa} that 
\begin{equation}\label{ujsarok}
\cornl{\csonkad}= \ulcover{{\cornl D}} \text.
\end{equation}
It follows from Lemma~\ref{lemmagrkn} that $\csonkad\in\ssdset {n-1}$. From \eqref{ujsarok} we obtain that 
\begin{equation}\label{kdqWie}
\text{$\csonkad\in \ssdset{n-1}$ determines $D$.} 
\end{equation}
Hence $|\ssdpozset n|\leq |\ssdset{n-1}|=\numssd{n-1}$. Combining this with $|\ssdnulset n|=\numssd{n-1}$ and $\ssdset n=\ssdnulset n\mathrel{\dot\cup} \ssdpozset n$, where $\dot\cup$ stands for disjoint union, we obtain 
\eqref{krket}.
\end{proof}                

Next, let 
\[W(n)=\ssdset{n-1} - \set{\csonkad: D\in \ssdpozset n}\text.
\]
This is the ``wrong'' set from our perspective since $W(n)=\varnothing$, which is far from reality, would turn inequality \eqref{krket} into an equality. Fortunately, this set is relatively small by the following lemma.
The upper integer part of a real number $r$ is denoted $\upintp x$, for example, $\upintgy 2=2$.

\begin{lemma}\label{wsetsmall} If $4\leq n$, then 
$\,\displaystyle{|W(n)|\leq \sum_{j=2}^{n+1-\upintgy {n-1}} \numssd j }$.
\end{lemma}

\begin{proof}
First we show that
\begin{equation}\label{wsetdescr}
W(n)=\set{E\in\ssdset{n-1}: \cornl E\text{ is a coatom of }E}\text. 
\end{equation}
The $\subseteq$ inclusion is clear from \eqref{sarokhely}, \eqref{hosszmarad}, and \eqref{ujsarok}. These facts together with Lemma~\ref{lemmagrkn} also imply the reverse inclusion since by adding a new cover to $\cornl E$, to be positioned to the left of $\leftb E$,  we obtain a slim, semimodular diagram $D$ such that $\csonkad=E$.

It follows from Lemma~\ref{wlalattnoredlMa} that
no down-going chain starting at $\leftb E$ can branch out. Thus  
\begin{equation}\label{ottaClc}
\ideal{{\cornl E}}\subseteq \leftb E\text{  
and }
\ideal{{\cornl E}}\text{ is a chain.}
\end{equation}
Since $\cornl E$ is a coatom, we have that
\begin{equation}\label{msiszH}
\text{with the notation }\cissor E=E\setminus\ideal{\cornl E}, \quad |\cissor E|=|E|-\length E\text.
\end{equation}
Clearly, $\cissor E$ is a join-subsemilattice of $E$ since it is an order-filter. To prove that
\begin{equation}\label{cissesubl}
\cissor E\text{ is a slim, semimodular diagram,}
\end{equation}
assume that $x,y\in\cissor E-\set 1$.
We want to show that $x\wedge y$, taken in $E$, belongs to $\cissor E$.  Let $x_0$ and $y_0$ be the smallest element of $\leftb E\cap \ideal x$ and $\leftb E\cap \ideal y$, respectively. Since $x_0,y_0\in \leftb E\cap\bigl(\ideal{\cornl E}-\set{\cornl E} \bigr)$, \eqref{saaqusio} implies that $x_0$ and $y_0$ are meet-reducible. 
Hence they have exactly two covers by \cite[Lemma 2]{r:czg-sch-visual}.
Let $x_1$ and $y_1$ denote the cover of $x_0$ and $y_0$, respectively, that do not belong to $\leftb E$, and let $\uucover x$ and $\uucover y$ be the respective covers belonging to $\leftb E$.  By the choice of $x_0$, we have that $\uucover x\not\leq x$, whence $x_1\leq x$. Similarly, $y_1\leq y$.  Since $\leftb E$ is a chain and the case $x_0=y_0$ will turn out to be trivial, we can assume that $x_0<y_0$. We know that $x_1\not\leq y_0$ since otherwise $x_1$ would belong to $\leftb E$ by \eqref{ottaClc}. Using semimodularity, we obtain that $x_1\vee y_0\succ y_0$. Since $y_0$ has only two covers by \cite[Lemma 2]{r:czg-sch-visual} and 
$x_1\leq \uucover y$ would imply $x_1\in \leftb E$ by \eqref{ottaClc}, it follows that 
$x_1\vee y_0 =y_1$. Hence $x_1\leq y$, $x_1\leq 
x$, and $x_1\in\cissor E$ implies that $x\wedge y$ belong to (the order filter) $\cissor E$. 
Thus $\cissor E$ is (to be more precise, determines) a sublattice of (the lattice determined by) $E$. The semimodularity of $\cissor E$ follows from Lemma~\ref{lemmagrkn}.
This proves \eqref{cissesubl}. 

By \eqref{msiszH}, \eqref{cissesubl}, by a trivial argument,
\begin{equation}\label{isEhwt}
\cissor E\in \ssdset{n-\length E}\text{ and } \cissor E\text{ determines } E\text.
\end{equation}

Next, we have to determine what values $h=\length E$ can take. Clearly, $h\leq |E|-1=n-2$.  There are various ways to check that $|E|\leq(1+\length E)^2= (1+h)^2$; this follows from  the main theorem of \cite{czgschperm}, and follows also from the proof of \cite[Corollary 2]{r:czg-sch-howtoderive}. Since now $|E|=n-1$, we obtain that 
$\upintgy {n-1} -1\leq h$. Therefore, combining  \eqref{cissesubl}, \eqref{isEhwt}, we obtain that 
\[W(n)\leq \sum_{h= \upintgy {n-1} -1 }^{n-2}\numssd{n-h}\text. 
\]
Substituting $j$ for $n-h$ we obtain our statement.
\end{proof}

We conclude this section by the following lemma.

\begin{lemma}\label{beSsclma}
$\displaystyle{
2\cdot\numssd{n-1}- \sum_{j=2}^{n+1-\upintgy {n-1}} \numssd j  \leq \numssd n\leq 2\cdot\numssd{n-1} } .$
\end{lemma}

\begin{proof}
By \eqref{kdqWie} and the definition of $W(n)$, we have that 
\begin{align*}
\numssd n&=|\ssdnulset n|+|\ssdpozset n|= \numssd{n-1} + |\ssdset{n-1}-W(n)|\cr
&=\numssd{n-1} + \numssd{n-1} - |W(n)|,
\end{align*}
and the statement follows from Lemma~\ref{wsetsmall} and \eqref{krket}.
\end{proof}

\section{Tools from Analysis at work} 
For $k\ge 2$, define $\kappa_k=\numssd k/\numssd{k-1}$. Since $\numssd{n-3}/\numssd{n-1}=1/(\gamma_{n-1}\gamma_{n-2})$, dividing the inequalities of Lemma~\ref{lMasieaeh} by $\numssd{n-1}$ we obtain that $1+1/(\kappa_{n-1}\kappa_{n-2})\leq \kappa_n\leq n$, for $n\geq 4$. Therefore, since $\kappa_k\leq 2$ also holds for $k\in\set{2,3}$ and $1+1/(2\cdot 2)=5/4$, we conclude that 
\begin{equation}\label{kpabeCSz}
5/4 \leq \kappa_n\leq 2\text{, \quad for }n\geq 4\text.
\end{equation}
Clearly, $\numssd{k-1}=\numssd k/\kappa_n\leq \frac 45\cdot \numssd k$ if $k\geq 4$. Thus, by iteration, we obtain that 
\begin{equation}\label{vkrot}
\numssd {k-j}\leq (4/5)^j\cdot \numssd k,\qquad\text{for }
j\in\mathbb N_0\text{ and }k\geq j+4\text.
\end{equation}
If $k\geq 5$, then using $\numssd k\geq \numssd 5\geq 3$ (actually, $\numssd 5=3$), we obtain that 
\begin{align}\label{hataka}
\begin{aligned}
\numssd 1&+\cdots+\numssd k=1+1+1+\numssd 4+\cdots+\numssd k\cr
&\leq 3+\numssd k\cdot \bigl( (4/5)^{k-4} +(4/5)^{k-5}+\cdots + (4/5)^{0}\bigr)\cr
&\leq \numssd k+ \numssd k\cdot 1/(1-4/5)= 6\numssd k\text.
\end{aligned}
\end{align}
Combining Lemma~\ref{beSsclma} with \eqref{hataka}  and \eqref{vkrot} we obtain that 
\begin{align*}
2 \numssd{n-1}&- 6\cdot(4/5)^{\upintgy {n-1}-2}\cdot \numssd{n-1}\leq 
\cr
& 2\numssd{n-1}- 6\numssd{n+1-\upintgy {n-1}}
\cr
& \leq \numssd n\leq 2 \numssd{n-1}\text. 
\end{align*}
Dividing the formula above by $2\numssd{n-1}$ and \eqref{kpabeCSz} by 2, we obtain that
\begin{equation}\label{kRotunk}
\max\bigl(5/8\,,\,  1-  3\cdot(4/5)^{\upintgy {n-1}-2}\bigr) \leq \kappa_n/2\leq 1,\quad\text{for }n\geq 5\text.
\end{equation}

Next, let us choose an integer $m\geq 5$, and define 
\[ z_0=z_0(m)=\min\bigl(3/8\, , \, 3\cdot(4/5)^{\upintgy {m-1}-2} \bigr)\text. 
\]

\begin{figure}
\centerline
{\includegraphics[scale=0.9]{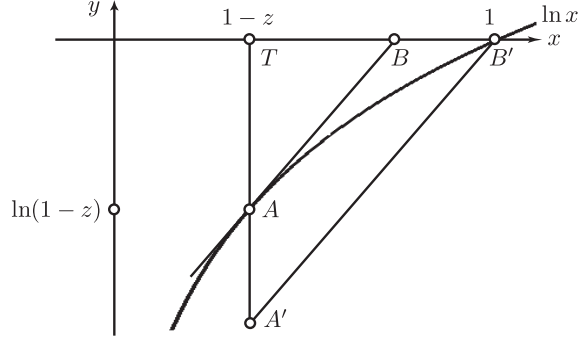}}
\caption{An illustration to Lemma~\ref{loglemma}  \label{fig:ket}}
\end{figure}

\begin{lemma}\label{loglemma}
If $0<z\leq z_0$, then 
\[-\ln(1-z)\leq z/(1-z)\leq z/(1-z_0)\text.\]
\end{lemma}

\begin{proof}
The statement follows from $\ln'(1-z)=1/(1-z)$ and the similarity of the triangle $ABT$ to the triangle $A'B'T$, see Figure~\ref{fig:ket}. 
\end{proof}

With the auxiliary steps made so far, we are ready to start the final argument.

\begin{proof}[Proof of Theorem~\ref{thmmain}]
For $n>m$, let 
\[
p_n=\prod_{j=m+1}^n (\kappa_j/2)\text.
\]
Clearly,
\begin{equation}\label{swieTuq}
\numssd n/2^n = p_n\cdot \numssd m/2^m\text.
\end{equation}
Hence it suffices to prove that the sequence $\set{p_n}$, that is $\set{p_n}_{n=m+1}^{\infty}$, is convergent. 
Let $s_n=-\ln p_n$, $\mu=3(1-z_0)^{-1}$, $\alpha=4/5$, and $\nu=5\mu/4=\mu/\alpha$. Then,
using \eqref{kRotunk} together with  Lemma~\ref{loglemma} at $\leq'$, 
\eqref{kRotunk} at $\leq^\ast$, and using that the function $f(x)=\alpha ^{\sqrt {x}}$ is  decreasing, we obtain that 
\allowdisplaybreaks{
\begin{align*}
0&< s_n =  \sum_{j=m+1}^n\bigl( -\ln(\kappa_j/2)\bigr)
\leq'  \sum_{j=m+1}^n (1- \kappa_j/2)/(1-z_0)\cr
&\leq^\ast \mu \cdot \sum_{j=m+1}^n \alpha ^{\upintgy {j-1}-2} 
\leq  \mu \cdot \sum_{j=m+1}^n \alpha ^{\sqrt{j-1}-1} 
\leq  \mu \cdot \sum_{k=m}^{n-1} \alpha ^{\sqrt{k}-1} 
\cr&
= 
 \nu \cdot \sum_{k=m}^{n-1} \alpha ^{\sqrt{k}}   
\leq   \nu \cdot \int_{x=m-1}^{n-1} \alpha ^{\sqrt{x}}dx \leq   \nu \cdot \bigl(F(\infty)-F(m-1)\bigr),
\end{align*}
}
where $F(x)$ is a primitive function of $f(x)$.
Let $\delta=-\ln\alpha=\ln{(5/4)}$. 
It is routine to check (by hand or by computer algebra) that, up to a constant summand,
\[
F(x) = - 2\cdot \delta ^{-2}\cdot (1+\delta \sqrt x) \cdot \alpha^{\sqrt x}\text.
\]
Clearly, $F(\infty)=\lim_{x\to\infty}F(x)=0$. This proves that the sequence $\set{s_n}$ converges; and so does $\set{p_n}=\set{e^{-s_n}}$ by the continuity of the exponential function. Therefore, since $\numssd m/2^m$ in \eqref{swieTuq} does not depend on $m$, we conclude  Theorem~\ref{thmmain}.
\end{proof}

\begin{remark}
We can approximate the constant  in Theorem~\ref{thmmain} as follows. Since
$ e^{-\nu \cdot (F(\infty)-F(m)) }   \leq  e^{-s_n}  =p_n\leq 1$ and, by \eqref{swieTuq}, 
$\konst =\lim_{n\to\infty}\bigl( p_n\numssd m /2^m \bigr)$, we obtain that 
\begin{equation}\label{beCsztC}  
e^{\nu  F(m) }\cdot \numssd m /2^m 
=
e^{-\nu \cdot (F(\infty)-F(m)) }\cdot \numssd m /2^m 
\leq  \konst \leq \numssd m /2^m 
\text.
\end{equation}
Unfortunately, our computing power
yields only a very rough  estimation. The largest $m$ such that $\numssd{50}$ is known is $m=50$, see \cite{r:czgandfour}. With $m=50$ and $\numssd m=\numssd{50}= 81\,287\,566\,224\,125$, it is a routine task to turn \eqref{beCsztC} into 
\[
 0.42\cdot 10^{-57}   \leq   C\leq 0.073\text{ .}
\]
We have reasons (but no proof) to believe that $0.023 \leq C\leq  0.073$, see the Maple worksheet (version V) available from the authors's home page.
\end{remark}

\begin{ackno}The author is indebted to Vilmos Totik for helpful discussions.
\end{ackno}

%
%
%


\begin{thebibliography}{99}


\bibitem{r:czgandfour}
    G. Cz\'edli, T. D\'ek\'any, L. Ozsv\'art, N. Szak\'acs, and B. Udvari: 
    On the number of slim, semimodular lattices, submitted.



\bibitem{r:czgolub}
    G. Cz\'edli, L. Ozsv\'art, B. Udvari,
    How many ways can two composition series intersect? Discrete Mathematics, submitted.



\bibitem{r:czg-sch-howtoderive}
   G. Cz\'edli,  E.\,T. Schmidt,
   How to derive finite semimodular lattices from distributive lattices?, Acta Mathematica Hungarica, \tbf{121} (2008) 277--282. 


\bibitem{r:czg-sch-JH}
   G. Cz\'edli,  E.\,T. Schmidt,
   The Jordan-H\"older theorem with uniqueness for groups and semimodular lattices,
   Algebra Universalis {66} (2011) 69--79.


\bibitem{r:czg-sch-visual}
   G. Cz\'edli,  E.\,T. Schmidt,
   Slim semimodular lattices. I. A visual approach,
   Order, published online May 5, 2011. (DOI: 10.1007/s11083-011-9215-3)


\bibitem{czgschperm}
   G. Cz\'edli,  E.\,T. Schmidt,
   Intersections of composition series in groups and slim semimodular lattices by permutations,
   submitted. 




\bibitem{r:erneheitzigreinhold}
    M. Ern\'e,  J. Heitzig,  J. Reinhold,
    On the number of distributive lattices,
    Electron. J. Combin. {9} (2002), no. 1, Research Paper 24, 23 pp. 
%


\bibitem{GK01}
G. Gr\"atzer, E. Knapp:
\emph{Notes on planar semimodular lattices I. Construction},
Acta Sci.\ Math. (Szeged). \tbf{73} (2007), 445--462.



\bibitem{r:heitzigreinhold}
  J. Heitzig,    J.  Reinhold, 
  Counting finite lattices,
  Algebra Universalis  {48} (2002) 43--53.


\bibitem{kellyrival}
  Kelly, D., Rival, I.: 
  Planar lattices. 
  Canad. J. Math. \tbf{27}, 636--665 (1975)

\bibitem{r:PawarWaphare}
  M.\,M. Pawar,    B.\,N.  Waphare, 
  Enumeration of nonisomorphic lattices with equal number of elements and edges,
  Indian J. Math. {45} (2003) 315--323.

 

\end{thebibliography}
\end{document}